\documentclass [11pt]{article}
\usepackage[]{amsmath,amssymb,amsfonts,latexsym,amsthm,enumerate,paralist}
\usepackage[numeric,initials,nobysame]{amsrefs}
\usepackage{bbm,paralist}
\date{}
\title{Longest cycles in sparse random digraphs}
\author{{Michael Krivelevich\thanks{School of Mathematical
Sciences, Raymond and Beverly Sackler Faculty of Exact Sciences, Tel
Aviv University, Tel Aviv 69978, Israel. Email address: {\tt
krivelev@post.tau.ac.il}. Research supported in part by USA-Israel
BSF Grant 2006-322 and by grant 1063/08
 from the Israel Science Foundation.}}
 \and {Eyal Lubetzky\thanks{Microsoft Research, One Microsoft Way, Redmond,
WA 98052-6399, USA. Email address: {\tt eyal@microsoft.com}.}}
\and {Benny Sudakov\thanks{Department of Mathematics, UCLA,  Los
Angeles, CA 90095, USA. Email: {\tt bsudakov@math.ucla.edu}.
Research supported in part by NSF CAREER award DMS-0812005 and by a USA-Israeli BSF grant.}}
}
\numberwithin{equation}{section}
\newtheorem{maintheorem}{Theorem}

\newtheorem{theorem}{Theorem}[section]
\newtheorem{lemma}[theorem]{Lemma}

\renewcommand{\epsilon}{\varepsilon}

\newcommand{\one}{\mathbbm{1}}
\newcommand{\E}{\mathbb{E}}
\renewcommand{\P}{\mathbb{P}}
\newcommand{\cC}{\mathcal{C}}
\newcommand{\cD}{\mathcal{D}}
\newcommand{\cF}{\mathcal{F}}
\newcommand{\cG}{\mathcal{G}}
\newcommand{\cP}{\mathcal{P}}
\newcommand{\cT}{\mathcal{T}}
\DeclareMathOperator{\var}{Var}

\DeclareMathOperator{\dist}{dist}

\DeclareMathOperator{\bin}{Bin}

\newtheoremstyle{upright}%
        {8pt plus2pt minus4pt}%
        {8pt plus2pt minus4pt}%
        {\upshape}%
        {}%
        {\bfseries}%
        {:}%
        {1em}%
        {}%
\theoremstyle{upright}

\newcommand{\ignore}[1]{}
\begin{document}

\maketitle

\begin{abstract}
Long paths and cycles in sparse random graphs and digraphs were
studied intensively in the 1980's. It was finally shown by Frieze
in 1986 that the random graph $\cG(n,p)$ with $p=c/n$ has a cycle on
at all but at most $(1+\epsilon)ce^{-c}n$ vertices with high
probability, where $\epsilon=\epsilon(c)\to 0$ as $c\to\infty$.
This estimate on the number of uncovered vertices is
essentially tight due to vertices of degree 1.
However, for the random digraph $\cD(n,p)$ no tight result was known and the best estimate was a factor
of $c/2$ away from the corresponding lower bound.
In this work we close this gap and show that
the random digraph $\cD(n,p)$ with $p=c/n$ has a cycle containing
all but $(2+\epsilon)e^{-c}n$ vertices w.h.p., where
$\epsilon=\epsilon(c)\to 0$ as $c\to\infty$. This is essentially
tight since w.h.p.\ such a random digraph contains $(2e^{-c}-o(1))n$ vertices with zero
in-degree or out-degree.
\end{abstract}

\section{Introduction}\label{sec:intro}
In this paper we consider long cycles in random directed graphs,
aiming to obtain estimates analogous to those derived for the
undirected case. Formally, a random graph $\cG(n,p)$ is a
probability space of all graphs with vertex set $[n]$, where each
pair of vertices $1\le i< j\le n$ is an edge of $G\sim\cG(n,p)$
independently and with probability $p$. The model of random directed
graphs $\cD(n,p)$ is defined as the probability space of all
directed graphs with vertex set $[n]$ (without loops and without
parallel edges, but possibly with anti-parallel edges), where each
ordered pair  $(i,j)$, with $1\le i\ne j\le n$, is  a directed edge
of  $D\sim \cD(n,p)$ independently and with probability $p$.

The existence of long paths and cycles in sparse random graphs was a
subject of very intensive study in the eighties. Ajtai, Koml\'os and
Szemer\'edi proved in \cite{AKS} that with high
probability\footnote{We say that a sequence of events $(A_n)$ in a
random (di)graph model occurs with high probability, or w.h.p.\ for
brevity, if the probability of $A_n$ tends to 1 as the number of
vertices $n$ tends to infinity.} in the random graph $\cG(n,c/n)$
there is a path of length $\alpha(c)n$, where $\alpha(c)>0$ for
$c>1$ and $\lim_{c\rightarrow\infty} \alpha(c)=1$; a similar but
somewhat weaker result was proved independently by Fernandez de la
Vega~\cite{Fd}. Then the attention has shifted to estimating the
asymptotic behavior of the number of vertices uncovered by a longest
path/cycle. Improving upon prior results of Bollob\'as~\cite{Bollobas} and Bollob\'as, Fenner and Frieze~\cite{BFF}, Frieze
has finally settled this problem:
He showed in~\cite{Frieze} that w.h.p.\ $G\sim\cG(n,c/n)$
contains a cycle covering all but at most $(1+\epsilon)ce^{-c}n$
vertices, where $\lim_{c\rightarrow\infty}\epsilon(c)=0$. This
estimate is easily seen to be asymptotically tight as $\cG(n,c/n)$ w.h.p.\ contains $(1+o(1))ce^{-c}n$ vertices of degree at most 1,
all of which have to be missed by a cycle.

For random directed graphs the situation appears to be more
complicated. This is to be expected as the research experience of
many years has shown that problems related to long paths and cycles
in directed (random) graphs are usually much more challenging than
their undirected counterparts. In the aforementioned paper \cite{Frieze} Frieze further established that
w.h.p.\ $\cD(n,p)$ contains a cycle covering all but at most
$(1+\epsilon)ce^{-c}n$ vertices, where $\lim_{c\to\infty}\epsilon(c)=0$.
 This result was derived by appealing to a general theorem of McDiarmid
\cite{McDiarmid79}, coupling between events in $\cG(n,p)$ and in
$\cD(n,p)$. Unlike in the undirected case, the above estimate on the
number of vertices uncovered by a longest cycle is no longer
asymptotically tight --- the unavoidable loss in the directed
case are vertices of in-degree  or out-degree zero, whose number is
easily seen to be asymptotic to $2e^{-c}n$.

In this paper we close the gap left by Frieze's work and obtain an
asymptotically optimal result about longest cycles in sparse random
digraphs.

\begin{maintheorem}\label{thm-1}
Let $D\sim \cD(n,p)$ be a random digraph with edge probability
$p=c/n$ for fixed $c>1$. Then w.h.p.\ $D$ contains a directed cycle
that covers all but at most $(2+\epsilon)e^{-c}\,n$ vertices, where
$\epsilon=\epsilon(c)\to 0$ as $c\to\infty$, and this is asymptotically
tight as w.h.p.\ $(2e^{-c}-o(1)) n$ vertices of $D$ have zero
in-degree or out-degree.
\end{maintheorem}

The proof of the theorem is given in the next section. In certain
similarity to Frieze's argument in \cite{Frieze} we proceed by first
filtering out vertices of zero in-degree or out-degree as well as some
vertices close to them. The so obtained digraph typically retains
all but a negligible fraction of the vertices of positive in-degrees and out-degrees; it is
then upgraded to another random digraph, containing an almost
spanning cycle, by sprinkling a few more random directed edges.

Before we embark into the technicalities of the proof, we provide
its outline, aiming to help the reader to parse the proof's details.

The proof has two components/stages: filtering and factoring. The
filtering stage aims to filter out vertices of in- or out-degree
zero and possibly some other vertices and to produce an induced
subgraph $D_0$ of $D\sim \cD(n,p)$, containing most of the vertices
of positive degree; moreover, $D_0$ is constructed in a way making
it rather straightforward to show that it typically contains a
factor of directed cycles. In order to produce $D_0$, we define the
following iterative process. Let $Y=\{v:d_D(v)=0\}$, and let $Z=\{v:
d_D(v)\le 3\}$, where $d_D(v)=\min\{ d_D^+(v),\, d_D^-(v)\}$. We
start with $X=\emptyset$, and then for $k\ge 1$ we obtain $X_k$ by
including vertices not in $\bigcup_{i<k} X_i$, lying on a short path
connecting two vertices $x,y\in \bigcup_{i<k}X_i \cup Z$. We repeat
this process till it stabilizes and set $X=\bigcup_{k}X_k$. Finally,
the subgraph $D_0$ is defined by $D_0=D[V-(X\cup Y)]$. Observe that
for a given vertex $v$ the probability that $d_D(v)$ is at most
some absolute constant (independent of $c$) is at most $\operatorname{poly}(c)e^{-c}$.
Thus, for a given $v$ the probability of having two vertices of
small degree at a constant distance from $v$ is $\operatorname{poly}(c)e^{-2c}$.
Hence, we can expect to eventually have $|X|\le \operatorname{poly}(c)e^{-2c}n$,
and this is indeed what we prove. To facilitate the proof, we first
get rid of short cycles (of length $O((1/c)\log n)$) in the
underlying undirected graph $G$ of $D$ --- they typically touch very
few vertices. Analyzing the filtering process in a large girth graph
is easier --- for every $v\in X_k$ there should be an evidence for
its association with $X_k$ in the form of a tree $T_v$ rooted at
$v$, of prescribed order, depth and with $\ell$ leaves, where
$k\le\ell\le 2^k$. This is proven in Lemma~\ref{lem-D'}. Using this
lemma we can bound the size of the set $X_k$ (or rather of a set
$X_k'$ closely related to $X_k$ and defined through an analogous
filtering process) by the number of labeled rooted trees meeting
these requirements. This is done by first truncating unusually deep
trees (Lemma~\ref{lem-Tk-empty}) and then bounding from above
the expected number of trees of bounded depth. The final argument 
invokes martingales to show the concentration of the corresponding
random variable around its mean and to bound its upper tail. All
this is done in Theorem~\ref{thm-X}.

The factoring stage takes the induced subgraph $D_0$, the output of
the filtering stage, as an input. By Theorem~\ref{thm-X} we know
that with high probability $D_0$ contains all but a suitably small part of the
vertices of $D$ of positive degree. Moreover,  one can prove 
(Lemma~\ref{lem-low-deg-D0}) that all vertices in $D_0$ have positive
degree, and in addition every two vertices $u,v$ with
$d_{D_0}(u),d_{D_0}(v)\le 2$ are at undirected distance at least 5.
We then form an auxiliary bipartite graph $H_0$ with parts $L,R$,
corresponding to two copies of the vertices of $D_0$, where an edge
$(x,y)\in  D_0$ becomes an edge $x_Ly_R\in E(H_0)$. It is quite easy
to see that the existence of a perfect matching in $H_0$ implies the
existence of a spanning subgraph of $D_0$ composed of directed
cycles. The probable existence of a perfect matching in $H_0$ is
shown in Lemma~\ref{lem-hall} using Hall's condition and standard
density/expansion arguments for random (di)graphs. The next step is
to trade  the factor of directed cycles in $D_0$ for one nearly
spanning cycle using extra random edges, about $O(n/\sqrt{\log n})$
of them -- a negligible quantity easily absorbed into the original
random digraph. This is done using rather standard random graph
arguments and extremal statements guaranteeing the existence of a
long cycle in a highly connected digraph (see Lemma~\ref{lem-BKS}).
The factoring stage is treated in Theorem~\ref{thm-D0}.

The next section contains the full details of the proof of the main result, and is followed by concluding remarks in Section~\ref{sec:concl}.

\section{Proof of main result}\label{sec:proof}

\subsection{Filtering and factoring}\label{sec:main}
If $D$ is a directed graph we use the notation $d_D(v)$ to denote
$\min\{ d_D^+(v),\, d_D^-(v)\}$. Similarly, we let $N_D(v) =
N_D^+(v) \cup N_D^-(v)$ and in both cases may omit the subscript $D$
when there is no danger of confusion.

For an undirected graph $G$ and a special subset of its vertices $Z$
we define a filtering process which produces a sequence $\{X_k\}$ of
disjoint subsets of the vertices  as follows:
\begin{align}
X_0 &= \emptyset\,,\nonumber\\
  X_{k} &= \left\{ v\notin \bigcup_{j<k}X_j :\, \begin{tabular}{l}$\exists \; x,y \in \big( \bigcup_{j<k}X_j\big) \cup Z$ such that $x\neq y$ and \\
  $v$ is on a path of length $l\leq 4$ between $x,y$, i.e.\\
$v\in\{x=u_0,u_1,\ldots,u_l=y\}$ with $u_i u_{i+1} \in E(G)$.\end{tabular}\right\}\mbox{ for $k \geq 1$}\,,\nonumber\\
X &= \bigcup_k X_k\,.
\label{eq-X-def}
\end{align}
The first ingredient in the proof is showing that w.h.p., once we
filter the set $X$ from the graph along with the vertices with zero
in/out degree, the remaining vertices may be factored into large
cycles and thereafter combined into a single long cycle while losing
only a negligible number of vertices in the process. This is shown
in the next theorem whose proof appears in
Section~\ref{sec:factoring}.
\begin{theorem}\label{thm-D0}
  Let $D\sim \cD(n,p)$ where $p=\frac{c}n$ for $c>1$ fixed and let $G$ be the undirected underlying graph of $D$.
  Let $Y=\{v : d_D(v) = 0\}$, $Z=\{ v : d_D(v) \leq 3\}$, and set $X(G,Z)$ as in~\eqref{eq-X-def}. Let $D_0$ be the induced subgraph of $D$ on
$V(D) \setminus (X \cup Y)$ and let $D_0'$ be its union with a random digraph $\cD(|D_0|,(n\sqrt{\log n})^{-1})$. If $|D_0|>n/5$ then w.h.p.\ $D'_0$ contains a directed cycle on $|D_0|-o(n)$ vertices.
\end{theorem}
The following theorem, which we prove in Section~\ref{sec:filtering}, estimates the size of the filtered subset $X$ w.r.t.\ vertices of low in/out degree in $D$.
\begin{theorem}\label{thm-X}
Let $D\sim \cD(n,p)$ be a random digraph with edge probability $p=c/n$ for fixed $c>1$. Let $Z = \{v : d_D(v) \leq 3\}$ and define $X=X(G,Z)$ as in~\eqref{eq-X-def} where $G$ is the undirected
underlying graph of $D$. If $c$ is sufficiently large then with high probability $|X| \leq (2c)^{10} e^{-2c} n$.
\end{theorem}
From the above two theorems we can immediately derive our main result.
\begin{proof}
  [\textbf{\emph{Proof of Theorem~\ref{thm-1}}}]
Let $Y=\{v : d_{D}(v) = 0\}$. Note that w.h.p.\ $|Y| =
(2e^{-c}+o(1))n$ since $d^+(v),d^-(v) \sim \bin(n-1,c/n)$.
  For a sufficiently large $c$ we obtain from Theorem~\ref{thm-X} that w.h.p.\ $|X \cup Y| \leq (2e^{-c} + (2c)^{10} e^{-2c} +o(1))n$.
In particular, for large $c$ and $n$ we have that w.h.p.\ $D_0$, the
induced subgraph on $V(D) \setminus (X\cup Y)$, has at least $n/5$
vertices (with room to spare) and we deduce from
Theorem~\ref{thm-D0} that w.h.p.\ $\cD(n,p')$ has a cycle missing at
most $|X\cup Y|+o(n)$ vertices, where $ p' = (c/n) + (n\sqrt{\log
n})^{-1} = (c+o(1))/n$. This establishes the required result for a
choice of, say, $\epsilon(c) = 2(2c)^{10}e^{-2c}$, which makes up for
$|X|$ with an extra factor of $2$ that readily absorbs the additive
$o(n)$-term in $|X\cup Y|$ as well as the $o(1)$-term in $p'$.
\end{proof}

\subsection{Long cycles in the filtered graph}\label{sec:factoring}
  To prove Theorem~\ref{thm-D0} we first need to establish several properties of the graph $D_0$ stemming from the definition of $X$ and the geometry of the random digraph $D$.
\begin{lemma}\label{lem-low-deg-D0}
Let $D_0$ be the induced subgraph on $V(D) \setminus (X \cup Y)$ and let $G_0$ be its undirected underlying graph. Then
\begin{compactenum}[(i)]
  \item\label{it-min-deg} Every $u\in D_0$ has $d_{D_0}(u) \geq 1$.
  \item\label{it-low-deg} Every $u,v\in D_0$ with $d_{D_0}(u),d_{D_0}(v) \leq 2 $ have $\dist_{G_0}(u,v)\geq 5$.
\end{compactenum}
\end{lemma}
\begin{proof}
To prove Part~\eqref{it-min-deg} assume that some $u\in V(D_0)$ has $d_{D_0}(u)=0$ and assume without loss of generality that $d^+_{D_0}(u)=0$.

First consider the case where $d^+_D(u) \geq 2$. Observe that in this case there exist distinct $x,y\notin D_0$ such that $(u,x),(u,y)\in E(D)$. Thus $u$ is on a path of length 2 between $x,y\in X\cup Y
\subset X \cup Z$, implying that $u \in X$ by definition and contradicting the fact that $u \in V(D_0)$. The case where $d^+_D(u) = 1$ is treated similarly: Here there is some vertex $v \notin V(D_0) $ such that $(u,v)\in E(D)$, $v \in X \cup Y \subset X \cup Z$, and furthermore $u\in Z$ by definition. Hence, $u$ is on a path of length $1$ between two distinct vertices $u\neq v$ in $X \cup Z$ and must thus also belong to $X$, in contradiction to the fact that $u\in V(D_0)$.

To prove Part~\eqref{it-low-deg} let $u,v$ be vertices satisfying $d_{D_0}(u) \leq 2$ and $d_{D_0}(v) \leq 2$. If $d_D(u) \geq 4$ then it necessarily lost at least $2$ neighboring (in/out) vertices in $X \cup Y$ and hence must also belong to $X$. We thus conclude that $d_D(u) \leq 3$ and similarly that $d_D(v) \leq 3$.

Let $G$ be the underlying undirected graph of $D$. Recalling that $u,v \in Z$ by the definition of $Z$, there cannot be a path of length at most $4$ between $u,v$ in $G$, as such a path would imply that $u,v$ must both belong to $X$. In particular, the induced subgraph $G_0\subset G$ also satisfies $\dist_{G_0}(u,v) \geq 5$, completing the proof.
\end{proof}

\begin{lemma}\label{lem-hall}
  Let $H_0$ be the undirected bipartite graph whose parts $(L,R)$ correspond each to the vertices of $D_0$ and where $x_L y_R \in E(H_0)$ iff $(x,y)\in E(D_0)$. Then w.h.p.\ $H_0$ has a perfect matching.
\end{lemma}
\begin{proof}
Recall that by Part~\eqref{it-min-deg} of Lemma~\ref{lem-low-deg-D0} there are no isolated vertices in $H_0$ (neither in $L$ nor in $R$). Furthermore,
by Part~\eqref{it-low-deg} of that lemma we know that if $x,y\in L$ have degree $1$ in $H_0$ then $N(x)\cap N(y) = \emptyset$ (otherwise $D_0$ would have two vertices with out-degree $1$ and an undirected
distance of at most $2$ between them) and similarly for $x,y\in R$ with degree $1$ in $H_0$. In other words, if we denote by $M_0$ the set of edges incident to degree $1$ vertices in $H_0$ then $M_0$
consists of vertex disjoint edges. Let $H_1$ denote the bipartite graph obtained by deleting the vertices of $M_0$ from $H_0$, i.e.\ $H_1 = H_0 \setminus V(M_0)$.
We now claim that $H_1$ has minimum degree at least 2. To see this, suppose that $d_{H_1}(u) \leq 1$ and argue as follows.

First, we must have $d_{H_0}(u) > 1$ otherwise $u \in V(M_0)$ and hence does not belong to $H_1$. If $d_{H_0}(u)= 2$ then there must be some $w \in V(M_0)$ such that $u w \in E(H_0)$.
In particular, either $w$ has degree $1$ in $H_0$ or it is a neighbor of such a vertex, and either way we have that there exists some degree-1 vertex $v\in H_0$ whose distance from $u$ is at most $2$.
The vertices corresponding to $u$ and $v$ in $D_0$ thus satisfy
$d_{D_0}(u)\leq 2$ and $d_{D_0}(v) \leq 1$ while the undirected distance between them is at most $2$, contradicting Part~\eqref{it-low-deg} of Lemma~\ref{lem-low-deg-D0}.

It thus remains to treat the case $d_{H_0}(u)\geq 3$. In this case $u$ has two neighbors $w_1,w_2\in V(M_0)$, giving rise to $v_1,v_2\in V(M_0)$ whose distance from $u$ is at most $2$ and with $d_{H_0}(v_1)=d_{H_0}(v_2)=1$. These correspond to two vertices $v_1,v_2$ in $D_0$ satisfying $d(v_1),d(v_2) \leq 1$ while the undirected distance between them is at most $4$, again contradicting Part~\eqref{it-low-deg} of Lemma~\ref{lem-low-deg-D0}.

We have thus obtained that $H_1$ has a minimum degree of $2$, and will now derive from this fact the existence of a perfect matching on $H_1$. It suffices to show that w.h.p.\ every set $S\subset V(H_1)\cap L$ of size at most $n/2$ has $|N(S)| \geq |S|$, as the same conclusion will carry by symmetry to all sets $S\subset V(H_1)\cap R$ of size at most $n/2$, which would in turn imply Hall's condition for sets $S \subset V(H_1)\cap L$ of size larger than $n/2$.

Let $S$ be a subset of $V(H_1) \cap L$ of size $s \leq  n/5$ let $T=N(S)$ in $H_1$ and assume that $T$ has size $t < s$. Identifying these vertices with those of the original digraph $D$ we have that $e(S,T) \geq 2 s$ by definition of $H_1$ and the fact that it has minimum degree $2$.

Moreover, observe that every $u \in S$ has at most $2$ neighbors in $V(D) \setminus T$. Indeed, since $T$ includes all the neighbors of $S$ corresponding to vertices of $H_1$, any other neighbor $v \in N_D^+(u)\setminus T$ must belong either to $X \cup Y$ or to the vertices corresponding to $V(M_0)$, and these satisfy:
\begin{compactenum}
  \item The vertex $u$ cannot have two distinct neighbors in $X \cup Y $ otherwise it would belong to $X$ by definition and hence would be excluded from $D_0$.
  \item The vertex $u$ cannot have two distinct neighbors in $V(M_0)$ otherwise there would exist some $x,y$ with $d_{D_0}(x)=d_{D_0}(y)=1$ and $\dist_{G_0}(x,y)\leq 4$, contradicting Part~\eqref{it-low-deg} of Lemma~\ref{lem-low-deg-D0}.
\end{compactenum}
Combining these arguments we conclude that $|N^+_D(u) \setminus T|\leq 2$, and note that for a given vertex $u$ and subset $T$ the probability of this event is at most
\[\P(\bin(n-t,p) \leq 2) \leq 3\binom{n-t}2 p^2 (1-p)^{n-t-2} \leq 2c^2 e^{-\frac45 n p}\,,\]
where the last inequality used the fact $t < s \leq n/5$ and holds for any sufficiently large $n$ as the $(1+O(p))$-factor was absorbed into the leading constant.
Further note that the event that $|N^+_D(u) \setminus T|\leq 2$ depends only on the edges from $u$ to $T$ and therefore for distinct vertices these events are independent.

At this point, the following straightforward first moment argument shows that w.h.p.\ $D$ cannot contain sets $S,T$ of the above sizes where $S$ has at least $2|S|$ edges going to $T$ and every $u\in S$ has at most $2$ edges going elsewhere. Indeed, the probability that such sets exist in $D$ for any given such $s,t$ is at most
\begin{align*}
  \binom{n}s \binom{n}t \binom{st}{2s} p^{2s} \left(2c^2 e^{-\frac45 np}\right)^s &\leq
  \left[\frac{en}{s} \Big(\frac{en}t\Big)^{t/s} \Big(\frac{et}2\Big)^2 \frac{c^2}{n^2} 2c^2 e^{-4pn/5} \right]^s
\leq \left[(e^4/2) c^4 e^{-4c/5} (t/n)^{1-\frac{t}s} \right]^s\\
 & \leq \left[(e^4/2) c^4 e^{-4c/5} \right]^s \frac{s-1}n =: \Delta(s,t)\,,
\end{align*}
where we used the inequality $\binom{a}b \leq (ea/b)^b$ and the fact that $t < s \leq n/5$.
For large enough $c$ we have $c^{4} e^{-4c/5} < 2e^{-5}$ and so
\[ \Delta(s,t) < e^{-s} (s-1)/n\,,\]
 and summing over the possible values of $s,t$ now gives that
\begin{align*}\sum_{t < s \leq n/5} \Delta(s,t) &= \sum_{t < s \leq 2 \log n} \Delta(s,t) +
\sum_{\substack{2 \log n \leq s \leq n/5 \\ t < s}} \Delta(s,t) \\
&\leq (2\log n)^2\frac{2\log n}n + (n/5)^2 e^{-10\log n} = o(1)\,.
\end{align*}
It remains to treat sets $S$ of size $n/5 < s \leq n/2$. Verifying Hall's condition for such sets follows immediately form that the fact that w.h.p.\ every two sets $S,T$ of size $n/5$ in $D$ have an edge from $S,T$, as the following calculation shows:
\begin{align*}
  \binom{n}{n/5}^2 (1-p)^{(n/5)^2} \leq \left[ (5e)^2 e^{-c/5}\right]^{n/5} = o(1)\,,
\end{align*}
where the last inequality holds for a sufficiently large $c$.
\end{proof}

We are now in a position to prove Theorem~\ref{thm-D0}.
\begin{proof}[\emph{\textbf{Proof of Theorem~\ref{thm-D0}}}]
The edges of the matching provided w.h.p.\ by Lemma~\ref{lem-hall} correspond to a spanning subgraph of $D_0$ comprised of disjoint directed cycles.
Our first step is to delete from $D_0$ all cycles of length less than $\frac12 \log_c n$. Note that
the number of vertices participating in such cycles in the original digraph $D$ is w.h.p.\ at most
\[ \sum_{l < \frac12 \log_c n}  n^l p^l \leq
(\log_c n) \sum_{l < \frac12 \log_c n} c^l = O(n^{1/2}\log n) = o(n)\,.\]
The remaining disjoint directed cycles, denoted by $C_1,\ldots,C_m$, thus contain $|D_0|-o(n)$ vertices.
Note
also that the total number of cycles $m$ satisfies:
$m\leq n/(\tfrac12\log_cn)=O(n/\log n)$.

Let $\cP=\{P_i\}_{i=1}^t$ be a maximum collection of vertex disjoint
directed paths, each of length precisely $\lceil \log^{0.9}n \rceil$, formed from the edges of
$\{C_i\}_{i=1}^m$. Since the number of vertices uncovered by $\cP$
in each $C_i$ is at most $\lfloor \log^{0.9}n\rfloor$ it follows that $\cP$ covers all but at
most $m \log^{0.9}n=O(n/\log^{0.1}n)$ vertices of $D_0$.
Furthermore, recalling that $n/5\leq |D_0|\leq n$, this implies that
\begin{equation}
  \label{eq-t-bounds}
  (\tfrac15 - o(1))n/\log^{0.9} n \leq t \leq n/\log^{0.9}n\,.
\end{equation}
For each path $P_j\in \cP$ define its prefix $A_j$ and suffix $B_j$
to be its first $L=\lfloor \log^{0.8}n\rfloor$ vertices and last $L$ vertices, respectively.
Consider now the digraph $D_1=\cD(|D_0|,(n\sqrt{\log n})^{-1})$. We
will use the edges of $D_1$ to weave most of the vertices covered by
$\cP$ into a long directed cycle, using the edges of the paths $P_j$
as a backbone. Define an auxiliary digraph $H$ where the vertex set $[t]$
corresponds to the paths $P_1,\ldots,P_t$ and $(i,j)\in E(H)$ iff $D_1$
contains an edge from $B_i$ to $A_j$. Notice that if $H$ contains a directed
cycle $C=(i_1,\ldots,i_l)$ then $D_0\cup D_1$ contains a directed
cycle of length at least $l(\log^{0.9}n-2\log^{0.8}n)$, obtained as
follows: Start at the last vertex of $A_{i_1}$ and proceed with the
vertices/edges along $P_{i_1}$; use an edge from $B_{i_1}$ to
$A_{i_2}$ to jump to $P_{i_2}$, then traverse the vertices/edges
along $P_{i_2}$ till an edge from $B_{i_2}$ to $A_{i_3}$ and so on;
finally use an edge from $B_{i_l}$ to $A_{i_1}$ and possibly some
edges of $A_{i_1}$ to close the cycle.

The digraph $H$ is a random digraph on $t = \Theta(n/\log^{0.9}n)$
vertices with edge probability $\rho$ that satisfies
$1-\rho=(1-(n\sqrt{\log n})^{-1})^{L^2}$, implying
that $\rho=(1+o(1))\frac{\log^{1.1}n}n$. We thus need to prove that such a
random digraph contains w.h.p.\ an almost spanning cycle. This is an
established fact, and here we derive it from the following lemma of
\cite{BKS}, whose short proof is included for completeness.

\begin{lemma}[\cite{BKS}]\label{lem-BKS}
Let $D=(V,E)$ be a directed graph on $t$ vertices in which for every
ordered pair $A,B$ of disjoint vertex subsets $A,B\subset V$ of size
$|A|=|B|=k$ there is an edge from $A$ to $B$. Then $D$ contains
a path of length at least $t-2k$ and a cycle of length at least
$t-4k$.
\end{lemma}
\begin{proof}
Fix an arbitrary order $\sigma$ on the vertices of $D$ and run the
DFS (Depth First Search) on $D$, guided by $\sigma$. The DFS
maintains three sets of vertices: Let $S$ be the set of vertices
which we have completed exploring, $T$ be the set of unvisited
vertices, and $U=V(G)-(S\cup T)$, where the vertices of $U$ are kept
in a stack (a last in, first out data structure). The DFS starts
with $S=U=\emptyset$ and $T= V(D)$, and at each stage moves a
vertex from $T$ to $U$ (an unvisited vertex with an incoming edge from the top of the stack $U$) or from $U$ to $S$ until eventually all
vertices are in $S$. As such, at some point in the course of the algorithm we must have
$|S|=|T|$; consider that point, and observe crucially that all the vertices in $U$ form
a directed path, and that there are no edges from $S$ to $T$. We
conclude that $|S|=|T|\leq k-1$, and therefore $|U|\geq t-2k + 2$, so
there is a directed path with $t-2k+1$ edges in $D$, as required. To
get a directed cycle of the desired length, take a path as above and use
a directed edge from its last $k$ vertices to its first $k$ vertices
to close a cycle.
\end{proof}

In order to apply the above lemma, take $k=\lfloor n/\log n\rfloor$
while recalling that $H\sim \cD(t,\rho)$ with $t$ satisfying~\eqref{eq-t-bounds} and $\rho=(1+o(1))\frac{\log^{1.1}n}n$. As $t \leq n/\log^{0.9}n$, the
probability that $H$ has two disjoint vertex sets $A,B$ of
cardinality $k$ each with no edges from $A$ to $B$ is at most
\[
\binom{t}{k}(1-\rho)^{k^2}\le \left[(et/k)\,e^{-\rho
k}\right]^k \leq \left[ (e+o(1))\big(\log n\big)^{0.1}\cdot
e^{-(1-o(1))\log^{0.1}n}\right]^k=o(1)\,,
\]
thus w.h.p.\ $H$ satisfies the conditions of Lemma
\ref{lem-BKS} and in turn it contains w.h.p.\ a cycle of length at least
$t-4k = (1-o(1))t$. As explained above, it follows that w.h.p.\ the digraph $D_0\cup D_1$ contains
a directed cycle covering all but
$O(k \log^{0.9}n)+O(n/\log^{0.1}n)=o(n)$ vertices, as required.
\end{proof}

\subsection{Controlling the effect of the filtering process}\label{sec:filtering}
\begin{proof}[\textbf{\emph{Proof of Theorem~\ref{thm-X}}}]

An important element in the proof would be to analyze the set $X$ with respect to a subgraph of $D\sim \cD(n,p)$ with a reasonably large undirected girth. To this end we need the following lemma.

\begin{lemma}\label{lem-D'}
Let $G$ be an undirected graph with girth $g$ and let $Z$ be a subset of its vertices. Define $X(G,Z)$ as in \eqref{eq-X-def}. For every $1\le k \leq g/8$ and $v \in X_k$ there is a tree $T_v \subset
G$ rooted at $v$ whose leaves are in $Z$ and interior vertices are in $\bigcup_{j<k} X_j$. Moreover, $T_v$ has at most $5(|T_v \cap Z|-1)$ vertices, at most $4k$ levels (including the root) and its number
of leaves $\ell$ satisfies $k < \ell \leq 2^k$.
\end{lemma}
\begin{proof}
We proceed by induction on $k$. For the induction base recall that if $v\in X_1$ then there are 2 vertices $x,y\in Z$ such that $v$ is on a path of length at most 4 between $x,y$ in $G$. Treat this path
as a tree $T_v$ rooted at $v$, and notice that it has $2$ leaves, at most $4$ levels including the root (as $\dist(v,x),\dist(v,y)\leq 3$) and the induced subgraph on it in $G$ is a tree by the girth
assumption on $G$. Furthermore, $T_v$ has at most $5 \leq 5(|T_v\cap Z|-1)$ vertices since $|T_v \cap Z| \geq 2$, thus satisfying the statement of the lemma.

Next, let $k > 1$ and let $v\in X_k$. Let $x,y \in Z \cup \bigcup_{j<k}X_j$ be the endpoints of a shortest path $P$ containing $v$ (by definition~\eqref{eq-X-def} the path $P$ has length at most $4$).
Suppose first that one of these vertices belongs to $Z$, i.e.\ without loss of generality $x\in Z$ whereas $y \in X_{k-1}$ (otherwise $v$ would have belonged to some $X_j$ with $j < k$). Define the tree
$T_v$ as a path $P_y$ of length $\dist_G(v,y)$ from the root $v$ to the sub-tree $T_y$, provided by the induction, together with another path $P_x$ of length $\dist_G(v,x)$ from $v$ to $x$. On one hand,
the paths $P_x,P_y$ are disjoint by definition, and furthermore, excluding their endpoints, their vertices do not belong to $Z \cup \bigcup_{j<k} X_j$ by the minimality of $P$ and in particular do not
belong to $T_y$. On the other hand, if the path $P_x$ does intersects $T_y$, which was
guaranteed to have at most $4(k-1)$ levels by induction, then together with $P_y$ they complete a cycle of length
at most $4(k-1)+4 < 4k \leq g/2$ in $G$ contradicting its girth assumption. We conclude that $T_v$ is indeed a tree, with at most $4(k-1)+4 = 4k$ levels including the root. Finally, $|T_v\cap Z| =
|T_y\cap Z| + 1$, hence
the induction hypothesis and the fact that $T_v$ adds at most $4$ vertices to $T_y$ together imply that
\[ |T_v| \leq |T_y|+4 \leq 5(|T_y\cap Z|-1) + 4 < 5(|T_v\cap Z|-1) \,.\]
It remains to treat the case where $x \in X_j$ for some $j < k$ while $y \in X_{k-1}$.
As before, if $T_x \cap T_y \neq \emptyset$ then together with the path $P$ we obtain a cycle of length at most $8(k-1)+4 < 8k \leq g$ in $G$, contradicting the girth assumption.
Otherwise, $|T_v \cap Z| = |T_x \cap Z|+ |T_y \cap Z|$ and so our hypothesis on $T_x,T_y$ gives that
\[ |T_v| \leq |T_x|+ |T_y|+3 \leq 5(|T_x\cap Z|-1)+5(|T_y\cap Z|-1) + 3 < 5(|T_v\cap Z|-1) \,.\]
Noting that the number of leaves $\ell(T_v)$ was either $\ell(T_y)+1$ or $\ell(T_x)+\ell(T_y)$ immediately implies that $k+1 \leq \ell(T_v) \leq 2^k$ and completes the proof of the lemma.
\end{proof}

Let $G$ denote the undirected underlying graph of $D\sim \cD(n,p)$, and define $\cC \subset V(G)$ to be comprised of all vertices that belong to cycles of length at most
\[ R = (20/c)\log n  \]
in $G$.
Since each edge appears in $G$ with probability at most $2p$ independently of other edges, the expected number of cycles of length $r$ in $G$ is at most $n^r (2p)^r / r$ and thus
\[ \E |\cC| \leq \sum_{r < R} (2c)^r \leq \frac{(2c)^{R}}{c-1} < n^{1/5}\,,\]
where we used the fact that $c/\log(2c) > 100$ for sufficiently large $c$. In particular, $|\cC| < n^{1/4}$ w.h.p.

Define $Z' = \cC \cup Z$ and let $D'$ be the graph obtained by deleting all inner edges between vertices of $\cC$ (i.e.\ all edges of the induced subgraph on $\cC$). Let $G'$ denote the undirected
underlying graph of $D'$ and for all $k$ let $X'_k$ denote the set $X_k(G',Z')$ defined via \eqref{eq-X-def}. A key observation is that
\begin{equation}
  \label{eq-X-X'-C}
  X \subset X' \cup \cC \,.
\end{equation}
To see this, recall that $X_0=\emptyset$ and assume by induction that
\[ \mbox{$\bigcup_{j < k} X_j \subset \left(\big(\bigcup_{j<k} X'_j\big) \cup \cC\right)$ for some $k \geq 1$}\,.\]
Let $v\in X_k\setminus \cC$.
Let $P$ be a shortest path containing $v$ in $G$ with endpoints $x,y\in \big(\bigcup_{j<k}X_j\big) \cup Z$. By the definition of $X_k$ and the minimality of $P$ we know that $P$ has $1 \leq l \leq 4$ edges and none of its interior vertices belongs to $\big(\bigcup_{j<k} X_j\big) \cup Z$.
Consider the two sub-paths from $v$ to $x,y$ (one of which is possibly empty) and let $x',y'$ be the first vertices on these respective paths that belong to $\big(\bigcup_{j<k}X_j\big) \cup \cC\cup Z$.
Since $x,y$ clearly belong to this set, this defines a sub-path $P'$ of length $1 \leq l' \leq l \leq 4$ that contains $v$ in $G$. Crucially, since $v\notin \cC$ the path $P'$ has no interior vertices in $\cC$ and therefore all of its edges belong to $G'$. Finally, the induction hypothesis ensures that $x',y'\in \big(\bigcup_{j<k}X'_j\big) \cup Z'$ and we conclude that $v \in \big(\bigcup_{j\leq k} X'_j\big) \cup Z'$, completing the induction.

Our next goal is to provide an upper bound on $X'$ which is linear in $n$ (with a suitably small coefficient), absorbing the negligible contribution to it from vertices in $\cC$.
Observe that by definition the girth of $G'$ is larger than $R = (20/c)\log n$ and set
\begin{equation}
  \label{eq-K-def}
  K = \left\lfloor (2/c) \log n\right\rfloor\,.
\end{equation}
Invoking Lemma~\ref{lem-D'} w.r.t.\ $G'$, for each $k \leq K$ we can bound $|X_k'|$ from above by $|\cT'_k|$ where
\begin{equation}
  \label{eq-Tk'-def}
  \cT'_k = \left\{ T \subset G' \,:\,
\begin{array}{l}
\mbox{labeled rooted tree with $\ell$ leaves, $k<\ell\leq 2^k$, all belonging to $Z'$,}\\
\mbox{a total of $t$ vertices for $t \leq 5(\ell-1)$ and at most $4k$ levels.}
\end{array} \right\}\,.
\end{equation}
In particular, we will be able to assert that $X'_K = \emptyset$ by showing that $\cT'_K$ is empty, as the next lemma establishes.
\begin{lemma}\label{lem-Tk-empty}
Set $K$ as in~\eqref{eq-K-def}. With high probability $X'_K = \emptyset$.
\end{lemma}
\begin{proof}
In what follows let $L(T)$ denote the set of leaves of a tree $T$ and recall that if $T \in \cT'_k$ then $L(T) \subset Z' = \cC \cup Z$ by definition.

Let $\cT^*_k$ be the set of all trees in $\cT'_k$ where at least $\ell-1$ of the leaves belong to $Z$. We have $\binom{n}t$ choices for the vertices of $T \in \cT^*_k$ on $t$ vertices, and the well-known Cayley formula asserts that the number of labeled rooted trees on $t$ vertices is $t^{t-1}$. The probability that a given labeled tree on $t$ vertices is in $G$ (an upper bound on the probability it belongs to $G' \subset G$) is exactly $(2p)^{t-1}$. Finally, if $u\in L(T) \cap Z$ then by definition $d_G(u) \leq 3$ and in particular $d_{G\setminus T}(u) \leq 3$. Crucially, the events $\{ d_{G\setminus T}(u) \leq 3\}$ for $u \in L(T)$ are mutually independent as well as independent of all the interior edges of $T$ (accounted for in the probability that $T \subset G$). Altogether, for all $k \leq K$,
\begin{align}
\E |\cT^*_k| &\leq \sum_{\ell = k+1}^{2^k} \sum_{t \leq 5(\ell-1)}
\binom{n}t t^{t-1} (2p)^{t-1} \ell \left(2\P\big(\bin(n-t,p)\leq 3\big)\right)^{\ell-1} \,.\nonumber\\
&\leq \sum_{\ell = k+1}^{2^k} \sum_{t \leq 5(\ell-1)}
e\left(2ec\right)^{t-1} \left(2\P\big(\bin(n-t,p)\leq 3\big)\right)^{\ell-1} n\,,
  \label{eq-T*-exp-prel}
\end{align}
where in the last inequality we used the facts that $\binom{n}t \leq (en/t)^t$ and $t > \ell$.
If $c$ is sufficiently large then $n-t=(1-o(1))n$ as $t < 5 \cdot 2^K =o(n)$ and in particular
$\P\big(\bin(n-t,p)\leq 3\big) \leq \frac{1}{2} c^3 e^{-c}$.
Plugging this in \eqref{eq-T*-exp-prel} gives that for sufficiently large $n$,
\begin{align}
\E |\cT^*_k| &\leq \sum_{\ell = k+1}^{2^k} \sum_{t \leq 5(\ell-1)}
 (2ec)^{t} \left(c^3 e^{-c}\right)^{\ell-1}n \nonumber\\
  &\leq
 \sum_{\ell = k+1}^{2^k}
 \left((2e)^5 c^8 e^{-c}\right)^{\ell-1}n
\leq n e^{-\frac34 c k}
  \,,\label{eq-T*-exp}
\end{align}
where the last inequality holds for large enough $c$ and $n$. Substituting $k=K=\left\lfloor (2/c) \log n\right\rfloor$ now gives
\begin{align*}
\E |\cT^*_K| &\leq n e^{-\frac34 c K} = O(n^{-1/2}) = o(1)\,,
\end{align*}
hence w.h.p.\ $\cT^*_K = \emptyset$.

Now consider $T \in \cT'_K \setminus \cT^*_K$. Here there exist distinct $u_i,u_j \in L(T) \cap \cC$. As $T$ has at most $4K$ levels and connects $u_i,u_j$ in $G'$ (where the inner edges between the vertices of $\cC$ are absent) this implies the existence of a subgraph $F \subset G$ with
$m$ vertices and at least $m+1$ edges such that
\[ m\leq 2R + 8K + 2 \leq (60/c)\log n\]
(accounting for $u_i$ and $u_j$, a path of length at most $8K$ between them and up to 2 cycles in $\cC$, with the last inequality holding for large enough $c$). When $c$ is sufficiently large, the probability that such a graph $F$ belongs to $G$ is at most
\begin{align}
   \label{eq-m-excess}
   \binom{n}m \binom{\binom{m}2}{m+1}(2p)^{m+1} \leq
\left(\frac{en}{m}\right)^m \left(\frac{em}{2}\right)^{m+1}(2c/n)^{m+1} \leq \frac{m}{n} \big(e^2c\big)^{m+1}
=  O(1/\sqrt{n}) = o(1)
 \end{align}
implying that $\cT'_K \setminus \cT^*_K$, and hence also $\cT'_K$, is w.h.p.\ empty. By~\eqref{eq-Tk'-def} and the remark following that definition it now follows that $X_K' = \emptyset$ w.h.p., as required.
\end{proof}
It remains to estimate $|\cup_{k<K} X_k'|$. To this end, let $B(\cC,R/2)$ be the set of all vertices whose undirected distance from $\cC$ in $G$ is less than $R/2 = (10/c)\log n$. Consider some $v\in X_k'$
for some $k \leq K$, let $T_v$ be the corresponding tree provided by Lemma~\ref{lem-D'} and suppose first that some leaf $u$ in $T_v$ belongs to $\cC$ (recall that every leaf of $T_v$ is in $\cC \cup Z$
by~\eqref{eq-Tk'-def}). Since by definition $T_v$ has at most $4k \leq 4K \leq (8/c)\log n$ levels it follows that $T_v \subset B(\cC,R/2)$ and in particular $v \in B(\cC,R/2)$.
Due to this argument, if we let $\cT_k$ denote the set of rooted trees in $\cT'_k$ where \emph{all} leaves belong to $Z$ and let $Y_k$ denote the number of vertices serving as roots of such trees, i.e.
\begin{align*}
  \cT_k &= \left\{ T \subset G' \,:\,
\begin{array}{l}
\mbox{labeled rooted tree with $\ell$ leaves, $k<\ell\leq 2^k$, all belonging to $Z$,}\\
\mbox{a total of $t$ vertices for $t \leq 5(\ell-1)$ and at most $4k$ levels.}
\end{array} \right\}\\
Y_k &= \#\big\{ v \in V(G) \,:\,\mbox{$v$ is the root of $T$ for some $T \in \cT_k$}\big\}
\end{align*}
(notice that clearly $Y_k \leq |\cT_k|$ for any $k$), then
\[ |\cup_{k<K} X'_k| \leq |B(\cC,R/2)| + \sum_{k<K} Y_k \,.\]
To estimate the size of $B(\cC,R/2)$ observe that each vertex $v$ in this set corresponds to a graph on $m < 3R/2$ vertices and at least $m$ edges. We can therefore repeat the calculation in~\eqref{eq-m-excess} to get that
\[ \E |B(\cC,R/2)| \leq  \sum_{m<3R/2} \binom{n}m \binom{\binom{m}2}{m}(2p)^{m}
\leq \sum_{m<3R/2} \big(e^2c\big)^m< \sqrt{n} \,,\]
where the last inequality is valid for large $c$. In particular, $|B(\cC,R/2)|<n^{3/4}$ w.h.p.\ and it remains to estimate $\sum_{k<K} Y_k$.

Consider $|\cT_1|$, counting rooted labeled trees in $G$ with $2$ leaves (i.e.\ paths with a distinguished vertex) and at most $5$ vertices and where both leaves are in $Z$. Conditioned on the existence of a given labeled path $P$ in $G$, the probability that its endpoints are in $Z$ is less than the probability that each endpoint has an at most 3 in-neighbors or at most 3 out-neighbors in $D\setminus P$. Altogether,
\begin{align*}
\E Y_1 &\leq \E |\cT_1| \leq \sum_{2 \leq t \leq 5} t n^{t} (2p)^{t-1} \left(2\P\big(\bin(n-t,p)\leq 3\big)\right)^2 \\
&\leq 4 \cdot 5 (2c)^{4} \left(2\P\big(\bin(n-5,p)\leq 3\big)\right)^2  n
\leq 20 (2c)^4 \left(c^3 e^{-c}\right)^2 n
< 600\, c^{10} e^{-2c} n\,,
\end{align*}
where the last inequality holds for sufficiently large $n$.

Next examine $|\cT_k|$ for $2 \leq k < K$, which counts trees with at most $4k$ levels, $\ell \in\{k+1,\ldots,2^k\}$ leaves and a total of $t\leq 5(\ell-1)$ vertices, where all leaves are in $Z$. The calculation in~\eqref{eq-T*-exp-prel},\eqref{eq-T*-exp}, with the single change that now all leaves (rather than $\ell-1$) belong to $Z$, yields
\begin{equation}
  \label{eq-E-Yk}
  \E Y_k \leq  \E |\cT_k| \leq \sum_{\ell=k+1}^{2^k}\left((2e)^5 c^8 e^{-c}\right)^{\ell} n \leq e^{-\frac34 c (k+1)} n\,,
\end{equation}
and combining the above inequalities we deduce that for large enough $c$ and $n$ we have
\begin{equation}
  \label{eq-sum-k<K-exp}
  \sum_{k<K} \E Y_k \leq 1000\, c^{10} e^{-2c} n\,.
\end{equation}
To assess the deviation of the $Y_k$'s from their mean, set $K_0 = \lfloor \log \log n \rfloor $ and observe that
\[\sum_{K_0 \leq k < K} \E Y_k \leq 2 e^{-\frac34 c K_0}n <  n/\log^2 n\,,\]
with the last inequality easily holding for $c$ large. Applying Markov's inequality we deduce that
\begin{equation}
   \label{eq-sum-Yk-k>K0}
   \P\bigg(\sum_{K_0 \leq k < K} Y_k \geq n/\log n\bigg) \leq  1/\log n = o(1)\,.
 \end{equation}
It remains to estimate the $Y_k$'s for $k < K_0$. To this end, define $Y'_k$ to be the number of roots of trees $T \in \cT_k$ such that every vertex in $T$ has degree less than $\log^2 n$ in $G$:
\[ Y'_k = \#\big\{ v \in V(G) \,:\,\mbox{$v$ is the root of $T$ for some $T \in \cT_k$ and $d(u)<\log^2 n$ for all $u \in T$}\big\}\,.\]
Recall that the underlying graph $G$ is obtained from $D$ by erasing its edge directions. Therefore, $G$ itself is a random undirected graph $\cG(n,p')$ with edge probability $p'=1-(1-p)^2=(1+o(1))2p$. As we will formally state later, $G \sim \cG(n,p')$ has maximum degree less than $\log^2 n$ except with extremely low (super-polynomial) probability, and so $Y_k = Y'_k$ w.h.p. We will show that $Y'_k$ is concentrated about its mean and then use it to derive concentration for $Y_k$.

Let $(M_t)$ be the edge-exposure Doob's martingale for $D$; that is, let $e_1,\ldots,e_{\binom{n}2}$ be an arbitrary ordering of the edges of the complete graph on $n$ vertices and set $M_t = \E\left[ Y_k \mid \cF_t\right]$ where
$\cF_t$ is the $\sigma$-algebra corresponding to revealing the indicators $\{ \one_{\{e_i \in E(D)\}} : i \leq t\}$.
We are interested in bounds on the increments of the martingale $(M_t)$ in $L^\infty$ and $L^2$.

Consider the effect of modifying one of the indicators $\one_{\{e \in E\}}$; clearly this can create or destroy a tree $T \in \cT_k$ only if that tree includes an endpoint of $e$ as one of its vertices. Since $Y'_k$ counts roots of such trees where every vertex has degree less than $\log^2 n$ and by the definition of $\cT_k$ each such tree has at most $4k$ levels (including the root), it follows that modifying $e$ can alter $Y'_k$ by at most $(\log^2 n)^{4k}$. In other words, $Y'_k$ is $B$-Lipschitz as a function of the edges of $D$, where
\[ B = (\log^2 n)^{4k} \leq (\log^2 n)^{4K_0} \leq \exp\left(8 (\log\log n)^2 \right) = n^{o(1)}\,.\]
It is a well-known (and easy to show) corollary that in this case $|M_{t+1}-M_t|\leq B$ for all $t$ (see, e.g.~\cite{AS} for the standard coupling argument deriving this for Doob's martingale of Lipschitz functions).

Now assume that we have exposed $\one_{\{e_1\in E\}},\ldots,\one_{\{e_t\in E\}}$ and are about to reveal whether or not $e_{t+1} \in E$. We wish to bound $\var(M_{t+1} \mid \cF_t)$. If we let $\theta = \E[ Y'_k \mid \cF_t,\, e_{t+1} \notin E]$ then the shifted variable $Q = M_{t+1} - \theta$ satisfies $\P(Q \neq 0\mid \cF_t) = \P(e_{t+1} \in E \mid \cF_t) =  p'\le 2p$ whereas $|Q| \leq B$ by the assumption that $\P(|M_{t+1}-M_t|\leq B)=1$ (in fact, even more precisely, one has $|Q| \leq B$ due to the $B$-Lipschitz property of $Y'_k$). Thus,
\[ \var\left( M_{t+1} \mid \cF_t\right) = \var(Q\mid \cF_t) \leq 2p(B)^2 = n^{-1+o(1)}\,.\]
and we conclude that for some $L = n^{1+o(1)}$ we have $\sum_{t} \var\left(M_{t} \mid \cF_{t-1}\right) \leq L$ with probability 1.

We are now in a position to apply the following large-deviation inequality which is a special case of a result of Freedman~\cite{Freedman}*{Theorem 1.6} (see also \cite{McDiarmid}*{Theorem 3.15}):
\begin{theorem}
Let $(S_0,S_1,\ldots,S_N)$ be a martingale with respect to the filter $(\cF_i)$.
Assume that $S_{i+1}-S_i \leq B$ for all $i$ and that $\sum_{i=1}^N \var(S_i\mid\cF_{i-1}) \leq L$ with probability 1 for some $L>0$.
Then for any $s>0$ we have $\P\left(\bigcup_{i=1}^N\{S_i \geq S_0 + s\}\right) \leq \exp\left[- \tfrac12 s^2 /(L + B s )\right]$.
\end{theorem}
Plugging in our estimate for $|M_{t+1}-M_t|$ and $\sum_t \var(M_t \mid \cF_{t-1})$ while recalling that by definition of the Doob martingale $M_0 = \E Y'_k$ while $M_{\binom{n}2} = Y'_k$ it now follows that
\[ \P(|Y'_k - \E Y'_k | > s) \leq 2 \exp\left[ -\tfrac12 s^2 /\big(n^{1+o(1)}+n^{o(1)}s\big)\right]\,, \]
and in particular
\begin{align}\label{eq-Y'k-concen}
\P(|Y'_k - \E Y'_k | > n^{3/4}) \leq \exp\big(n^{-1/2+o(1)}\big)\,.
\end{align}
To complete the proof, recall that the probability that any vertex in $G \sim \cG(n,p')$ would have degree at least $\log^2 n$ is at most
\[  n \P\left(\bin(n,2c/n)\geq \log^2 n\right) \leq n \exp\left(-c' \log^2 n\right) < n^{-10}\,,\]
where the last inequality holds for large enough $n$. In particular, $Y'_k = Y_k$ except with probability $n^{-10}$ and since by definition $0 \leq Y'_k \leq Y_k \leq n$ we further have $\E[Y_k] = \E[Y'_k] + O(n^{-9})$. Combining these inequalities with~\eqref{eq-Y'k-concen} now gives
\[ \P(|Y_k - \E Y_k | > 2n^{3/4}) \leq 2n^{-10}\,,\]
where the extra factors of 2 absorbed the $O(n^{-9})$ and $\exp(n^{-1/2+o(1)})$ error terms. In particular, taking a union bound over the $K_0 \leq \log\log n$ values of $k$ we deduce that w.h.p.
\[ \sum_{k < K_0} Y_k - \sum_{k<K_0} \E Y_k \leq 2n^{3/4}\log\log n\,.\]
Finally, combining this inequality with~\eqref{eq-sum-k<K-exp} and~\eqref{eq-sum-Yk-k>K0} we conclude that w.h.p.
\[ \sum_{k < K} Y_k \leq 1000 c^{10} e^{-2c} n + 2n^{3/4}\log\log n + n/\log n = (1000+o(1))c^{10} e^{-2c} n\,,\]
where the last inequality holds for large enough $n$. Together with the aforementioned bounds on $X$ in terms of $X'$ and in turn of $X'$ in terms of $\sum Y_k$ we conclude that w.h.p.
\[ |X| \leq |\cC| + |B(\cC,R/2)| + \sum_{k<K} Y_k < n^{1/4} + n^{3/4} + (1000+o(1))c^{10} e^{-2c} n \leq (2c)^{10} e^{-2c} n\,,
\]
where the last inequality is valid for any sufficiently large $n$, as required.
\end{proof}

\section{Concluding remarks}\label{sec:concl}
We have proved that a random directed graph $\cD(n,c/n)$ contains
with high probability a directed cycle including all but at most
$(2+\epsilon)e^{-c}n$ vertices, where $\epsilon=\epsilon(c)\to 0$ as
$c\to\infty$. In fact, our proof shows that the relative error term
$\epsilon(c)$ is exponentially small in $c$, namely $\epsilon(c)\le
\operatorname{poly}(c)e^{-c}$.
The main term in the result is asymptotically optimal as such a
random digraph typically contains  $(2e^{-c}-o(1))n$ vertices with
zero in-degree or out-degree.

It would be very interesting to derive accurate estimates for the
length of a longest cycle in $\cD(n,c/n)$ for small(er) values of
the constant $c$, starting perhaps as low as the threshold for the
appearance of a linear length cycle in such a random digraph. See
the related work~\cite{Luczak91} 
where {\L}uczak studied the length of 
the longest cycle in the undirected random graph near its critical window, 
showing lower and upper bounds that are tight up to a factor of $1+\log(3/2)\approx 1.41$.

Compared to the situation in undirected graphs, the toolkit
available for the case of directed graphs is rather poor at present,
thus making the progress in a variety of questions about directed
random and pseudo-random graphs much harder to achieve. In
particular, the absence of any form of a direct analogue of the
famed P\'osa's rotation-extension technique, widely applied for
undirected graphs, is felt throughout. It would be very useful to
derive some directed version of it.

In general, the field of random and pseudo-random directed graphs is
largely an uncharted territory, compared to the situation for the
undirected case. Although this is certainly partly due to its
relative difficulty, we believe enough knowledge and technology have
been accumulated now to start exploring it in a systematic way.
One recent such result is the paper~\cite{BKS2}, where global resilience type results with
respect to long cycles have been derived for sparse random and
pseudo-random directed graphs. It would be interesting to explore
further resilience type questions in directed graphs.

\vspace{0.2cm}
\noindent
{\bf Acknowledgment.}\,
A major part of this work was carried out when the
first and the third authors were visiting Microsoft Research at Redmond, WA. They
would like to thank the Theory Group at Microsoft Research
for hospitality and for creating a stimulating research environment.

\end{document}